\documentclass[a4paper,11pt]{amsart}
\usepackage{epsfig,latexsym,amsfonts,
amsmath,amssymb,dsfont,color}
\usepackage{graphicx}

\numberwithin{equation}{section}
\newtheorem{Theorem}{Theorem}[section]

\def\comment#1{
\,
\color{blue}
#1
\color{black}
\newline
}
\def\comment#1{}
\def\Hu0{{\mathcal H}^{1}\!+\!u_0}

\def\Rd{{\mathbb R}^d}
\def\bfn{{\bf n}}
\def\bfe{{\bf e}}
\def\bfv{{\bf v}}

\def\bfw{{\bf w}}

\def\bfb{{\bf b }}

\def\bfx{{\bf x}}
\def\bfN{{\bf N}}
\def\bfE{{\bf E}}
\def\bfT{{\bf T}}

\def\bfG{{\bf G}}

\def\bftau{{\boldsymbol\tau}}
\def\bfsigma{{\boldsymbol\sigma}}

\def\bfeta{{\boldsymbol\eta}}
\def\bfzeta{{\boldsymbol\zeta}}
\def\bfchi{{\boldsymbol\chi}}

\def\ed{\end{document}}

\def\cD{{\mathcal D}}

\def\INT{{\mathds I}}

\def\Frame#1{
\begin{tabular}{|c|} \hline
   \\
$\
#1
$
\\
\\ \hline
\end{tabular}
}
\def\dvg{\rm div}
\def\IntO{\int\limits_{\Omega}}

\def\Rd{\mathbb R^d}

\newcommand\be{\begin{eqnarray*}}
\newcommand\ee{\end{eqnarray*}}
\newcommand\ben{\begin{eqnarray}}
\newcommand\een{\end{eqnarray}}
\def\Frame#1{
\begin{tabular}{|c|} \hline
   \\
$\
#1
$
\\
\\ \hline
\end{tabular}
}

 \def\Mod#1{\left\vert #1 \right\vert}

\def\Nor2#1{\langle\!\langle #1 \rangle\!\rangle}
\def\mean#1{\left\{\!\!\left|\,#1\,\right|\!\!\right\}}

\def\wh{\widehat}
\def\wt{\widetilde}

\def\wh{\widehat}

\def\cD{\mathcal D}

\def\IntO {\int\limits_\Omega}

\def\cT{{\mathcal T}}
\def\cD{{\mathcal D}}

 \def\Mod#1{\left\vert #1 \right\vert}
 \def\wt{\widetilde}

\def\dvg{{\rm div}}

\def\IntO{\int\limits_\Omega}

\begin{document}
\title
{  Poinca\'e type inequalities
for  vector functions with zero mean normal traces on the boundary 
and applications to interpolation methods
}
\author{ S. Repin}
 \thanks{University of Jyv\"askyl\"a, Finland 
 and St. Petersburg Department of Steklov Institute of Mathematics of Russian Academy of Sciences }
\email{serepin@jyu.fi,\;repin@pdmi.ras.ru}
\dedicatory { Dedicated to Professor Yuri Kuznetsov on the occasion
     of his $70$th birthday}

\maketitle

\begin{abstract}\,
In the paper, we consider  inequalities of the Poincar\'e--Steklov type for subspaces of $H^1$-functions defined in a bounded
domain $\Omega\in \Rd$ with Lipschitz boundary
$\partial\Omega$. For scalar valued functions, the subspaces
are defined by zero mean condition on $\partial\Omega$ or on
a part of $\partial\Omega$ having positive $d-1$ measure. For vector valued functions,
zero mean conditions are imposed on components (e.g., normal components) of the function on certain  $d-1$ dimensional manifolds (e.g., on plane or curvilinear faces of $\partial\Omega$).
 We find explicit and simply
computable bounds of the respective constants for domains typically used
in finite element methods (triangles, quadrilaterals,
tetrahedrons, prisms, pyramids, and domains composed of them).
The second part of the paper discusses applications of the estimates
to interpolation of scalar and vector valued functions.
\end{abstract}


\section{Introduction}
\subsection{Classical Poincar\'e inequality}

{\sc H. Poincar\'e} \cite{Poincare1984} proved that
$L^2$ norms of functions with zero mean defined in
a bounded domain $\Omega$ with smooth boundary
$\partial\Omega$ are uniformly bounded by the $L^2$ norm of the gradient, i.e.,
\begin{eqnarray}
&&\label{1.1}
\|w\|_{2,\Omega}\leq C_{\rm P}(\Omega)
\|\nabla w\|_{2,\Omega},\qquad
\forall w\in \widetilde H^1(\Omega),
\end{eqnarray}
where
$$
\widetilde H^1(\Omega):=\left\{w\in H^1(\Omega)\,\mid\,\mean{w}_\Omega:=\frac{1}{|\Omega|}
\IntO w\,dx=0\right\}.
$$
Poincar\'e  also deduced the very first estimates of $C_{\rm P}$:
\ben
\label{1.2}
&C_{\rm P}(\Omega)\leq\frac34{\rm d}_\Omega,\quad {\rm d}_\Omega:={\rm diam}\, \Omega\quad &{\rm for}\; d=3\;\\
\label{1.3}
&\;C_{\rm P}(\Omega)\leq\,
\sqrt{\frac{7}{24}}{\rm d}_\Omega\approx 0.5401{\rm d}_\Omega\quad &{\rm for}\; d=2.
\een
For piecewise smooth domains the inequality (\ref{1.1}) (and a similar inequality for functions vanishing on the boundary) was independently established by {\sc V. Steklov} \cite{Steklov}, who proved
 that 
$C_{\rm P}=\lambda^{-\frac 12}$, where $\lambda$ is the smallest positive eigenvalue of the problem
\ben
\label{1.4}
&-\Delta u={\lambda} u&\mbox{in}\quad \Omega;\\
\label{1.5}
&\partial_{\bfn}u=0&\mbox{on}\quad \partial\Omega.
\een

Easily computable estimates of $C_{\rm P}(\Omega)$ are
known for {\em convex domains} in $\Rd$. An upper bound
\ben
\label{1.6}
&&\;C_{\rm P}(\Omega)\,\leq \,\frac{{\rm d}_\Omega}{\pi}\;\approx 0.3183\, {\rm d}_\Omega
\een
was established in {\sc L. E. Payne } and {\sc H. F.  Weinberger} \cite{PW} (notice that for  $d=2$ the upper bound (\ref{1.3})
found by Poincar\,e is not  far from the sharp estimate (\ref{1.6})).

A lower bound of $C_{\rm P}(\Omega)$
was derived
in 
 {\sc S. Y. Cheng} \cite{Cheng} (for $d=2$):
 \ben
 \label{1.7}
 C_{\rm P}(\Omega)\geq\,\frac{{\rm d}_\Omega}{2j_{0,1}}\;\approx
 0.2079\,{\rm d}_\Omega.
 \een
Here $j_{0,1} \approx 2.4048$ is the smallest
positive root of the Bessel function $J_0$.
%

For {\em isosceles triangles} an improvement
of the {upper bound} (\ref{1.6}) is presented in
{\sc R. S. Laugesen} and {\sc B. A. Siudeja} \cite{LaSi}
\ben
\label{1.8}
&&C_{\rm P}(\Omega)\leq\,
    \frac{d_\Omega}{j_{1,1}},  
    \een
where $j_{1,1} \approx3.8317$ is the smallest positive root of the Bessel
function  $J_1$.


Poincar\'e type inequalities also hold for $L^q$ norms
if $1\leq q<+\infty$. In
{\sc G. Acosta and R. Duran} (2003), it was
shown that for convex domains the constant
in $L_1$ Poincar\'e type inequality satisfies the estimate
\ben
\label{1.9}
\inf_{c\in {\mathbb R}}\|w-c\|_{L_1}\leq \frac{{\rm d}_\Omega}{2}\|\nabla w\|_{L_1}.
\een
Estimates of the constant for other $q$ can be found in
{\sc S.-K. Shua} and {\sc R. L. Wheeden} (2006) (also 
for convex domains).
%
\subsection{Poincar\'e type inequalities for functions
with zero mean boundary traces}
Inequalities similar to (\ref{1.1}) also hold for functions
with zero mean traces on the boundary (or
on a measurable part
$\Gamma\subset\partial\Omega$) such that
${\rm meas}_{(d-1)}\Gamma>0$.
For any
\begin{eqnarray*}
w\,\in\,\wt H^1_\Gamma(\Omega)=\Bigl\{w\in H^1(\Omega)\,\Bigl\vert\Bigr.\,\mean{w}_{\Gamma} :=\frac{1}{|\Gamma|}\int\limits_\Gamma w\,ds
=0\Bigr\},
\end{eqnarray*}
we have two estimates for the $L^2(\Omega)$ norm of $w$
\ben
\label{1.10}
\|w\|_{2,\Omega}\leq \;C_{\Gamma}(\Omega)\|\nabla w\|_{2,\Omega}
\een
and and for its trace on the $\Gamma$
\ben
\label{1.11}
\|w\|_{2,\Gamma}\leq\;
 C^{\rm Tr}_\Gamma(\Omega)\|\nabla w\|_{2,\Omega}.
\een
 Existence of positive constants $C_{\Gamma}(\Omega)$ and $C^{\rm Tr}_\Gamma(\Omega)$
is proved by standard compactness arguments. 
Inequality (\ref{1.10}) arises in analysis of certain physical phenomena
(the so called "sloshing"  frequencies, see
{\sc D. W. Fox} and {\sc J. R. Kuttler} \cite{Slosh1},  {\sc V. Kozlov} et al. \cite{Slosh2,Slosh3} and references therein).  In the paper  by I. {\sc Babuska} and {\sc A. K. Aziz} \cite{BabuskaAziz} it was used in proving sufficiency of the maximal
angle condition for finite element meshes with triangular elements.
Inequalities (\ref{1.10}) and (\ref{1.11}) can be useful in many other
cases, e.g., for nonconforming approximations, a posteriori
error estimates (see \cite{MatNeiRep,ReENUMATH,MaNeRe,ReGruyter}), and and advanced interpolation methods for scalar
and vector valued functions.
In this paper, we are mainly interested in the  inequality (\ref{1.10}) for functions
with zero mean on $\Gamma$.
For the sake of brevity, we will call 
it the {\em boundary Poincar\'e inequality}. 

Exact constants $C_\Gamma$ and $C^{\rm Tr}_\Gamma$ are known
only for a restricted number of "simple" domains.
Table 1 summarises some of the results presented in {\sc A. Nazarov and  S. Repin} \cite{NazRep}, which are related to such domains as
rectangle $\Pi_{h1\times h_2}:=(0,h_1)\times (0,h_2)$,
parallelepiped 
$\Pi_{h1\times h_2\times h_3}:=(0,h_1)\times (0,h_2)\times (0,h_3)$, and right triangle $\overline T_h:={\rm conv}\{(0,0),(h,0), (0,h)\}$.
\begin{table}[h]
\label{tab:exactconst}
\begin{tabular}{|c|c|c|c|}
 \hline
 &&&\\
$d$ &  $\Omega$ &        $ \Gamma$      &      $ C_\Gamma(\Omega)$\\
 \hline
 $2$ & $\Pi_{h1\times h_2}$           &  {\rm face}\,$x_1=0$
   &  $ c_1 \max\{2h_1;h_2\}$,\quad $c_1=1/\pi$ \\
    \hline
 $2$ & $\Pi_{h1\times h_2}$           &  $\partial\Omega$
   &  $ c_1\max\{h_1;h_2\}$ \\
    \hline
 $3$ & $\Pi_{h1\times h_2\times h_3}$
      &   {\rm face}\,$x_1=0$
         &   $ c_1\max\{2h_1;h_2;h_3\}$               \\
          \hline
$2$ &$T_h$          &    {\rm leg} &     
$ c_2h$, \quad $c_2=1/\zeta$,\;                       
 $\scriptstyle\zeta\approx 2.02876 $ \\
 \hline
  $2$&$T_h$
  & {\rm two\,legs}   &
  $ c_1h $           \\
   \hline
  $2$ &$T_h$ &  {\rm hypothenuse}
   & $ \sqrt{2}c_2h $ \\
 \hline
\end{tabular}
\vskip3pt
\caption{Sharp constants}
\label{tab}
\end{table}

In Section 2 we deduce
easily computable majorants of $C_\Gamma$ for {\em triangles, rectangles, tetrahedrons, polyhedrons, pyramides and prizmatic type domains}. These results yield interpolation estimates (and respective
constants) for interpolation of scalar valued functions on macrocells based
on mean values on faces. As a result, we can deduce interpolation estimates
for functions defined on meshes with very complicated (e.g., nonconvex) cells.

Section 3 is concerned with boundary Poincar\'e inequalities
for vector valued functions. 
Certainly, (\ref{1.10}) admits a straightforward extension
to vector fields. We consider more sophisticated forms where
zero mean conditions are
imposed on mean values of different components of a vector 
valued function $\bfv$ on different $d-1$ dimensional manifolds (which are assumed to be sufficiently regular). In particular, it suffices to impose zero
mean  conditions on normal components of $\bfv$ on $d$
Lipschitz manifolds (e.g., on $d$ faces lying on $\partial\Omega$).  Then,
\ben
\label{1.12}
\|\bfv\|_\Omega\,\leq\, {\mathds C}(\Omega,\Gamma_1,...,\Gamma_d)\|\nabla \bfv\|_\Omega.
\een
Theorem \ref{InqVector} proves (\ref{1.12}) by compactness
arguments. 
After that, we consider the case where the conditions
are imposed on normal components of a vector field on $d$ different
faces of polygonal domains in $\Rd$ and
deduce (\ref{1.12}) directly by
applying (\ref{1.10}) to normal components of the vector field. This method also
yields easily computable majorants of the constant $\mathds C$.
 
 The last part of the paper is devoted to interpolation of functions
 defined in a bounded Lipschitz domain   $\Omega\in \Rd$,
 which are based on mean values of the function (or of mean values
 of normal components) on some set $\Gamma\in {\mathbb R}^{d-1}$.  
 It should be noted that interpolation methods 
based on normal components of vector fields defined on edges
of finite elements
are widely used in numerical analysis of PDEs (see, e.g., \cite{BrFo,RoTo}). 
Raviart--Thomas (RT) type interpolation operators
and their properties for approximations on polyhedral meshes
has been deeply studied in papers
of {\sc D. Arnold, D. Boffi and R. Falk} \cite{ArBoFa2002,ArBoFa2005}, {\sc A. Bermudes} et. all
\cite{Bermudes} and  other publications. The respective interpolants
belong to the space $H(\Omega,\dvg)$.
Approximations of this type are often used in
 mixed and hybrid finite element methods (see, e.g., {\sc F. Brezzi} and {\sc M. Fortin} \cite{BrFo}, {\sc J. E. Roberts} and {\sc J.-M. Thomas}
 \cite{RoTo}, {\sc V.~Girault} and {\sc P.~A.~Raviart}
 \cite{GiRa}).

This paper is concerned with coarser  interpolation methods, which
provide only $L^2$ approximation
of fluxes (and $H^{-1}$ approximation for the divergence
what is sufficient for treating balance equations in a weak sense!).
Hopefully this type interpolation methods  could be useful
for numerical analysis of PDEs on highly distorted meshes. This
challenging problem has been studying for many years by {\sc Yu. Kuznetsov} and coauthors
(see \cite{KuzPro2010,KuzRep2003,Kuz2006,Kuz2011,Kuz2014,
Kuz2015} and other publications cited therein).   
Smooth (high order)  methods 
 are probably too difficult for the interpolation of vector valued functions on very irregular (distorted) meshes.
 Moreover, in the majority of cases smooth interpolants seem to be
 not really natural because exact solutions often have 
 a very restricted regularity
 and because efficient numerical procedures (offered, e.g.,
  by the above mentioned dual mixed and
 hybrid methods) operate with low order approximations for fluxes. If meshes are very
irregular, then it is convenient to apply approximations of the lowest
possible  order and respective numerical methods with minimal
regularity requirements. Poincar\'e type estimates for
functions with zero mean conditions on manifolds of the dimension
$d-1$ yield interpolants of exactly this type.  

In Section 4 it is proved
that in $\Omega$ the difference between $u$ and its interpolant  $\INT_\Gamma u$
is controlled by the norm of $\nabla u$ with a constant, which depends
on the maximal diameter 
of the cell (due to results of previous sections, realistic 
estimates the interpolation constants are known
for "typical" cells). Finally, we shortly discuss interpolation on meshes when a (global) domain $\Omega$ is decomposed into a collection of  local
subdomains (cells) $\Omega_i$.  Using cell interpolation operators,
we define the global interpolation operator $\INT_{\cT_h}$ and prove the
respective interpolation estimates for scalar and vector valued functions. The interpolation
method operates with minimal amount of interpolation parameters related to mean values on a certain
amount of faces and preserves mean values on faces (for scalar valued functions) and mean values of normal components (for vector valued functions).

\section{ Estimates of $C_\Gamma$ for typical mesh cells}

\subsection{Triangles}
Consider a nondegenerate triangle ABC (Fig. 1 left) where $\Gamma$
coincides with the side AC.
\label{sec:simplicial}
\begin{figure}[h]
\label{fig:triangleabc}
\centerline{\includegraphics[width=2.1in]{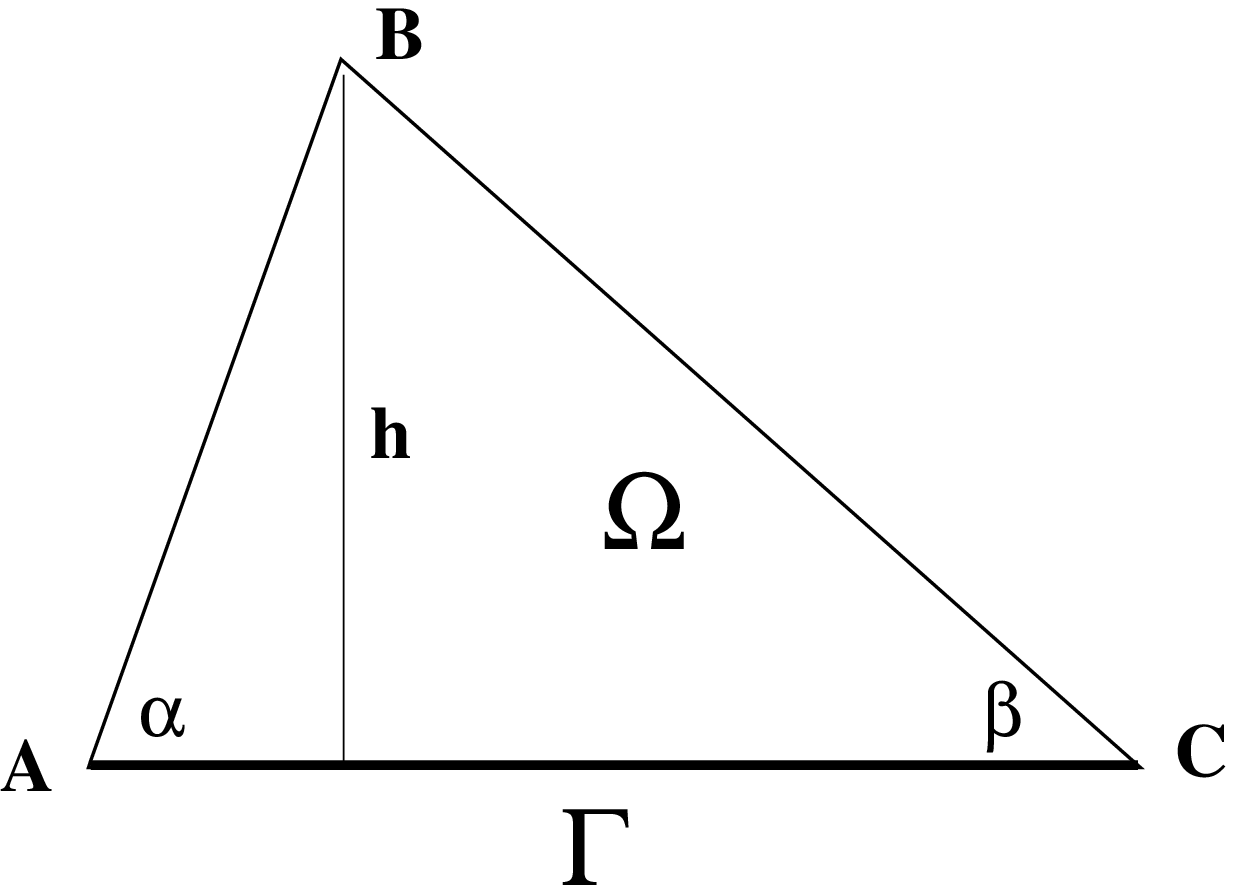}\qquad
\includegraphics[width=2.1in]{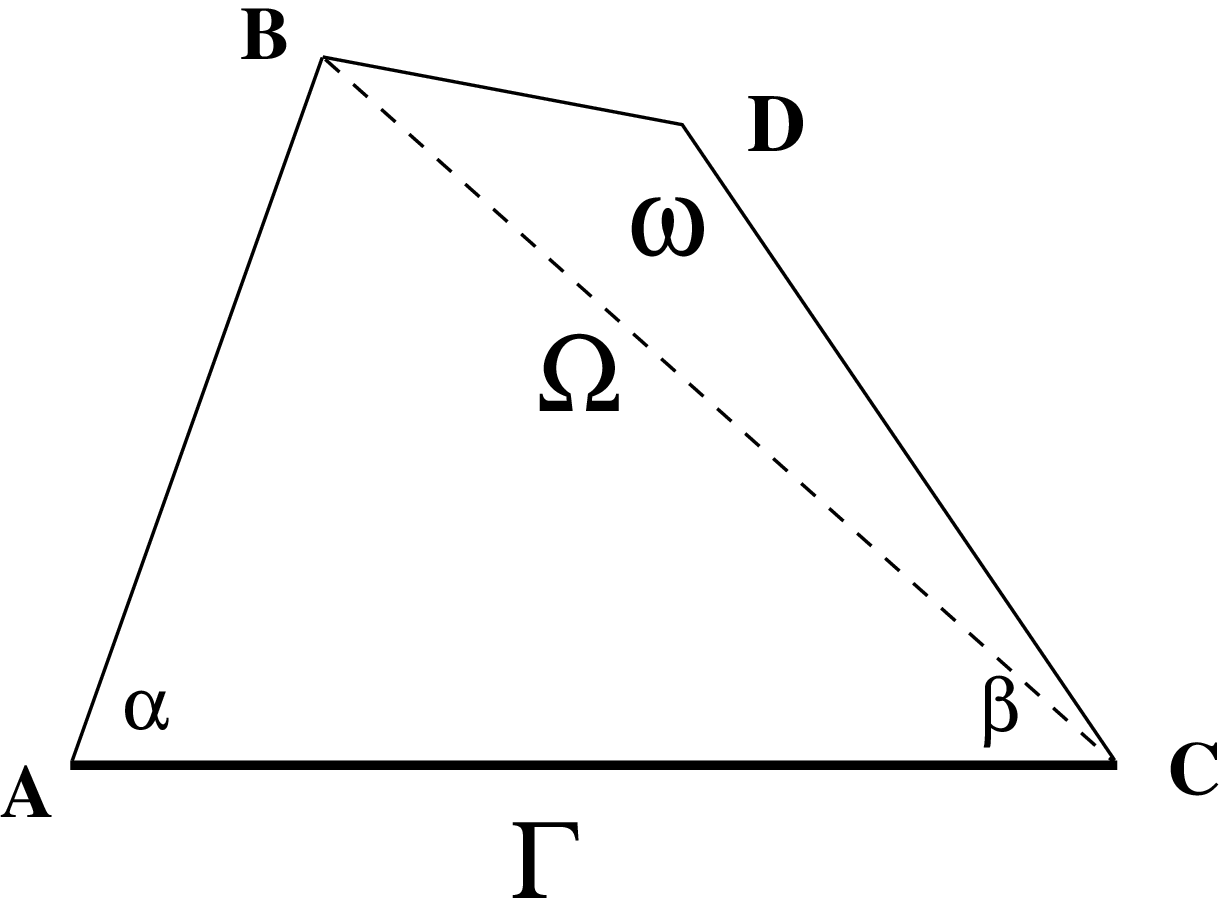}
}
\caption{Triangle and quadrilateral}
\end{figure}
\subsubsection{Majorant of $C_\Gamma$}
Our analysis is based upon the estimate
\ben
\label{simplex1}
C^2_\Gamma\leq\,C^2_{\rm P}+\frac{|\Omega|}{|\Gamma|^2}\inf\limits_{\tau\in Q(\Omega)}\|\bftau\|^2_{2,\Omega},
\een
which is a special form of the upper bound of $C_\Gamma$ derived in {\sc S. Repin}
\cite{ReRJNAMM}. Here
$
 Q(\Omega)$
 is a subset of $H(\Omega,\dvg)$ containing vector functions such that $\dvg\bftau=\frac{|\Gamma|}{|\Omega|}$,
 $\bftau\cdot\bfn=1\;{\rm on}\;\Gamma$, and  $\bftau\cdot\bfn=0\;{\rm on}\;\partial\Omega\setminus\Gamma$.
 We set
$\bftau$ as an affine field with values at the nodes A,B, and C 
$(-\cot\alpha,-1)$, $(0,0)$, and $(\cot\beta,-1)$, respectively.
In this case,
\be
\|\bftau\|^2_{2,\Omega}=
\frac13 |\Omega|(\frac32+\frac14\cot^2\alpha+\frac14\cot^2\beta+
\frac14(cot\beta-\cot\alpha)^2)
=\frac{|\Omega|}{6}\Sigma_{\alpha\beta}.
\ee
where 
$$
\Sigma_{\alpha\beta}=\cot^2\alpha+\cot^2\beta-\cot\alpha\cot\beta+3.
$$
Since $|\Omega|=\frac12 h|\Gamma|$, we see that
$\frac{|\Omega|^2}{|\Gamma|^2}=\frac{h^2}{4}$. In view
of (\ref{1.6}), the constant $C_{\rm P}$ is bounded from above by 
$\frac{d_\Omega}{\pi}$, where 
$d_\Omega=\max\{|AB|,|BC|,|CD|\}$, and we deduce an easily computable
bound
\ben
\label{simplex2}
C^2_\Gamma\leq\,C^2_{\rm P}+\frac{h^2\Sigma_{\alpha\beta}}{24}\,\leq\,\frac{{\rm d}^2_\Omega}{\pi^2}+\frac{h^2\Sigma_{\alpha\beta}}{24}.
\een
\comment{
If $\alpha,\beta\rightarrow \pi/2$, then $\Sigma_{\alpha\beta}\rightarrow\, 3$
and the upper bound tends to $h^2\left(\frac{1}{\pi^2}+\frac18\right)=h^2 0.226$.
Exact value tends to 
$h^2 0.173$.
}

We can represent $\Sigma_{\alpha\beta}$ in a somewhat different form
\be
\Sigma_{\alpha\beta}\,=\,
\frac{|AB|^2+|BC|^2+
\stackrel{\rightarrow}{AB}\cdot \stackrel{\rightarrow}{BC}}{h^2},
\ee
\comment{
\be
\Sigma_{\alpha\beta}\,=\,
\frac{1}{\sin^2\alpha}+\frac{1}{\sin^2\beta}+\frac
{\sin\alpha\sin\beta-\cos\alpha\cos\beta}{\sin\alpha\sin\beta}\\
=\frac{|AB|^2}{h^2}+\frac{|BC|^2}{h^2}-\frac
{|AB||BC|\cos(\alpha+\beta)}{h^2}\\
\color{blue}=\frac{1}{h^2}\left(|AB|^2+|BC|^2+|AB||BC|\cos(\pi-\alpha-\beta)\right)=\\
\frac{|AB|^2+|BC|^2+
\stackrel{\rightarrow}{BA}\cdot \stackrel{\rightarrow}{BC}}{h^2},
\ee
}
which yields the estimate
\ben
\label{simplex3}
C^2_\Gamma\leq\,\frac{{\rm d}^2_\Omega}{\pi^2}+\frac{|AB|^2+|BC|^2+
\stackrel{\rightarrow}{BA}\cdot \stackrel{\rightarrow}{BC}}{24}.
\een
\paragraph{ Example.} If $\alpha=\frac{\pi}{2}$, then ${\rm d}^2_\Omega=h^2+|\Gamma|^2$, $|\Gamma|=h\cot\beta$,\\ ${\rm d}^2_\Omega=h^2(1+\cot^2\beta)$ and
we obtain
\be
C_\Gamma\leq\,
h\sqrt{\frac{1}{\pi^2}+\frac{1}{8}+\cot^2\beta\left(\frac{1}{\pi^2}+\frac{1}{24}\right)}\approx 0.4757\,h\sqrt{1+0.6354\cot^2\beta}
\ee
In particular, for $\beta=\frac{\pi}{4}$, 
we obtain $C_{\Gamma}\leq 0.6083 h$ (exact constant for the riht triangle is $0.4929h$).

\subsubsection{Minorant of $C_\Gamma$}
A lower bound for $C_{\Gamma}$ follows from (\ref{1.7}) and
Irelations between $C_{\rm P}(\Omega)$ and $C_\Gamma(\Omega)$.
Any function in $\wt H^1_\Gamma(\Omega)$ can be represented
as $w-\mean{w}_\Gamma$, where $w\in H^1(\Omega)$. Hence, 
\be
(C_\Gamma(\Omega))^{-2}=
\inf\limits_{w\in H^1(\Omega)}
\frac{\IntO |\nabla w|^2\,dx}{\IntO |w-\mean{w}_\Gamma|^2\,dx}
\ee
and the constant $C_\Gamma(\Omega)$ can be defined as maximum
of $\|w-\mean{w}_\Gamma\|_{2,\Omega}$ for all $w\in H^1(\Omega)$ such that
$\|\nabla w\|_{2,\Omega}=1$.  Analogously, $C_{\rm P}$ can be defined as maximum of $\|w-\mean{w}_\Omega\|_{2,\Omega}$  over the same set of functions. Since
\be
\|w-\mean{w}_{\Gamma}\|_{2,\Omega}\geq\,
\inf\limits_{c\in {\mathbb R}}\|w-c\|_{2,\Omega}=\|w-\mean{w}_\Omega\|_{2,\Omega},
\ee
 we conclude
 that for any selection of $\Gamma$
 \ben
 \label{2.4}
 C_{\rm P}(\Omega)\;\leq\; C_\Gamma(\Omega).
 \een
From (\ref{1.7}) and (\ref{2.4}), it follows that
$
C_\Gamma\,\geq\, \frac12 \frac{{\rm d}_\Omega}{j_{0,1}}.
$
In particular, for
$\alpha=\frac{\pi}{2}$ we have 
$
C_\Gamma\geq\,
0.2079\,h\sqrt{1+\cot^2\beta}. 
$ 
\subsection{Quadrilaterals}
Using previous results, we deduce an estimate
of $C_\Omega$ for a quadrilateral ABCD (Fig. \ref{fig:triangleabc} right).
On $\Omega_1$ we set the same field $\tau$ as in the previous case
and set $\tau=0$ on $\Omega_2$. Let $\kappa^2=\frac{|\Omega_2|}{|\Omega_1|}$.
\comment{
Mean divergence is $\mean{\dvg\tau}_\Omega=\frac{|\Gamma|}{|\Omega|}$.
Notice that
\be
\int\limits_{\Omega_1}\left(\frac{|\Gamma|}{|\Omega_1|}-\frac{|\Gamma|}{|\Omega|}\right)^2\,dx=|\Gamma|^2\frac{|\Omega_2|^2}{|\Omega|^2|\Omega_1|},\quad
\int\limits_{\Omega_2}\frac{|\Gamma|^2}{|\Omega|^2}\,dx=
|\Omega_2|\frac{|\Gamma|^2}{|\Omega|^2}\\
\int\limits_{\Omega_1}+\int\limits_{\Omega_2}=\,\frac{|\Gamma|^2|\Omega_2|}{|\Omega|^2}\left(\frac{|\Omega_2|}{|\Omega_1|}-1\right)=
\,\frac{|\Gamma|^2|\Omega_2|}{|\Omega|\,|\Omega_1|}
=
\,\kappa^2\frac{|\Gamma|^2}{|\Omega|}
\ee
We have
\be
&&C^2_\Gamma\leq C^2_{\rm P}+\frac{|\Omega|}{|\Gamma|^2}
\left(||\tau||+C_{\rm P}\kappa\frac{|\Gamma|}{|\Omega|^{1/2}}\right)^2=
C^2_{\rm P}+\left(\frac{|\Omega|^{1/2}}{|\Gamma|}
||\tau||+\kappa C_{\rm P}\right)^2
\ee
}
Then,
\ben
\label{quadrilateral}
C^2_\Gamma\,\leq \,
C^2_{\rm P}+\left(\kappa\,C_{\rm P}+\frac{\Sigma^{1/2}_{\alpha\beta}|\Omega|}{\sqrt{6}\,|\Gamma|}\right)^2.
\een
Note that (\ref{quadrilateral}) also holds for more general cases in which $\Omega_2$
is a bounded Lipschitz domain having only one common boundary
with $\Omega_1$, which is $BC$.

\subsection{Tetrahedrons}
\label{subsec:tetrahedron}
Consider a tetrahedron OABC (Fig. 2 left), where $\Gamma$ is the triangle ABC which
lies in the plane $Ox_1x_2$.
\begin{figure}[h]
\label{fig:tetrahedron}
\centerline{\includegraphics[width=1.96in]{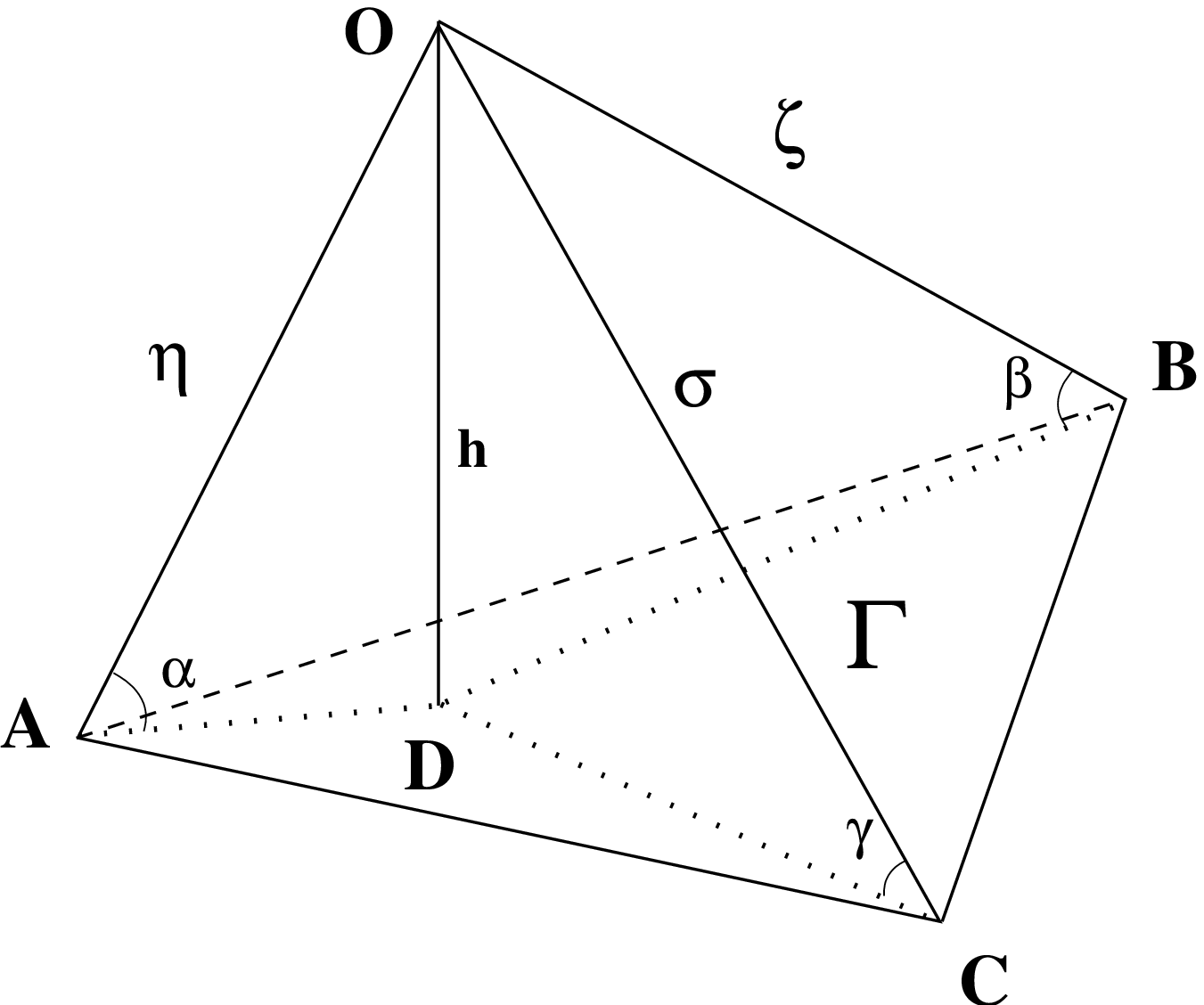}\quad
\includegraphics[width=1.2in]{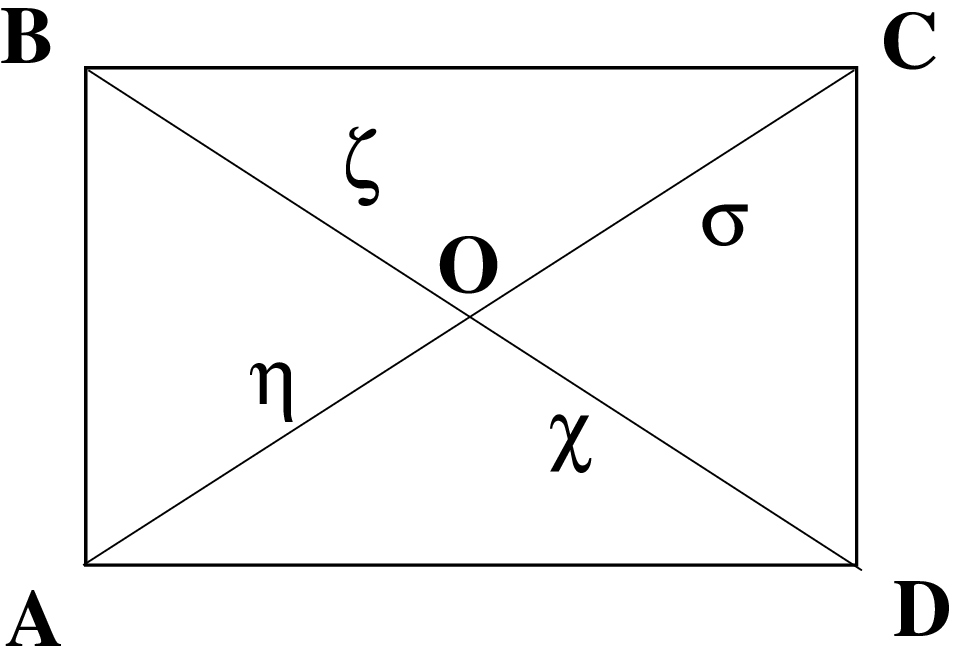}\quad
\includegraphics[width=1.2in]{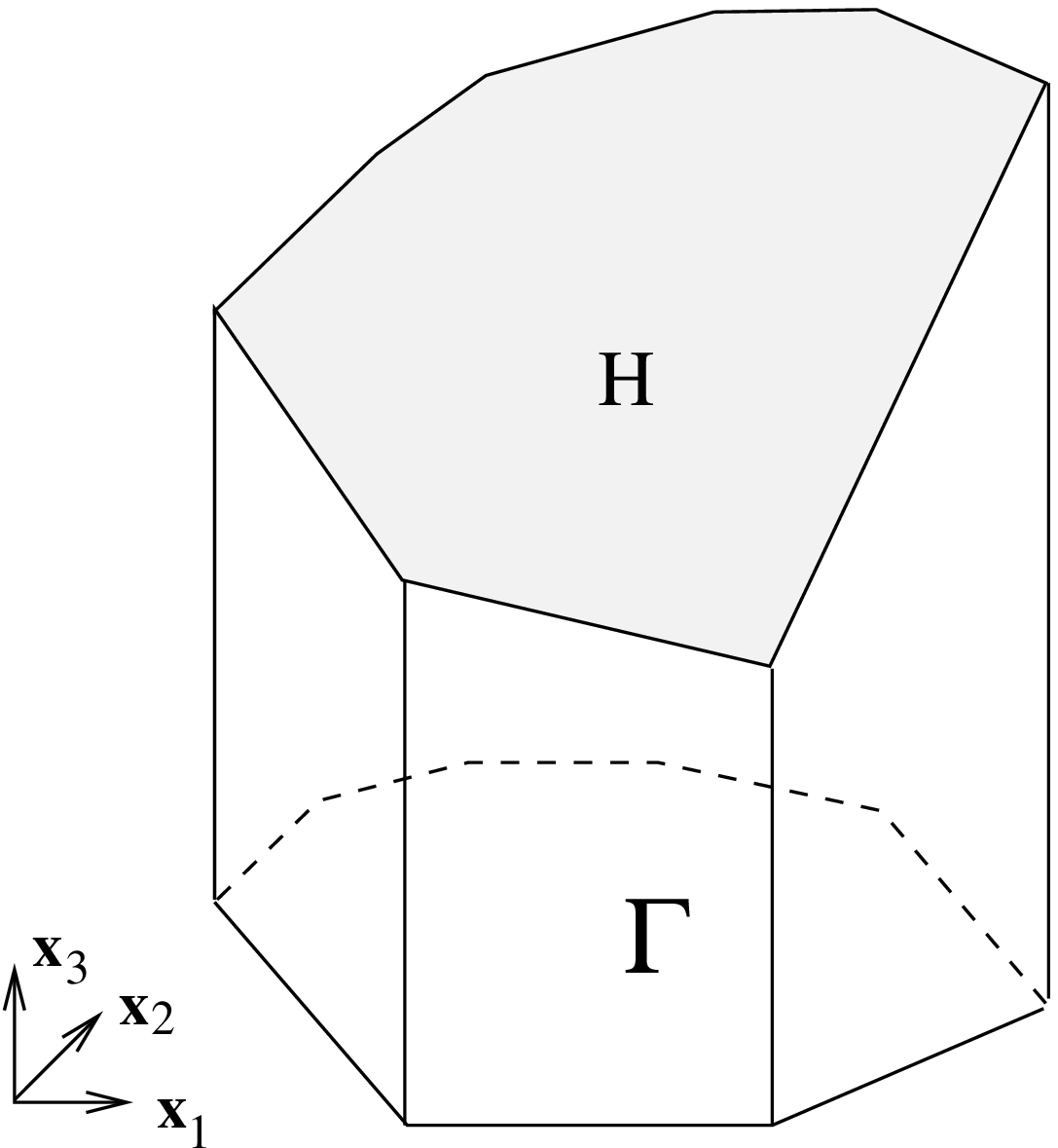}
}
\caption{Tetrahedron, pyramide, and prizm}
\end{figure}

%
%
At vertexes A, B, and C, we define three constant vectors
$$
\wh\bftau_A=\frac{  {\bfeta}}{|\bfeta|\sin\alpha},\quad
\wh\bftau_A=\frac{  {\bfzeta}}{|\bfzeta|\sin\beta},\quad {\rm and}\quad
\wh\bftau_A=\frac{  {\bfsigma}}{|\bfsigma|\sin\gamma}.
$$
 The vector
field $\bftau(x_1,x_2,x_3)$ is the affine field in $\Omega$ with zero value
at the vertex O. 
We compute
\be
\IntO |\tau|^2\,dx=\int\limits^h_0\left(\;\int\limits_{\omega(x_3)} |\tau(x_1,x_2,x_3)|^2dx_1dx_2\right)dx_3.
\ee
Notice that the cross section $\omega(x_3)$ associated with the height $x_3$ has the measure
$|\omega(x_3)|=\left(1-\frac{x_3}{h}\right)^2|\Gamma|$
and at the respective point $A^\prime$ on $OA$ (which third coordinate is $x_3$)
by linear proportion we have
$\bftau_{A^\prime}=\left(1-\frac{x_3}{h}\right)\wh\bftau_A$. Similar relations hold  for the
points $B^\prime$ and $C^\prime$ associated with the cross section on the height $x_3$.
For the internal integral we apply the Gaussian quadrature 
for $|\tau|^2=\tau^2_1+\tau^2_2+\tau^3_3$ 
and obtain
\comment{
\begin{multline*}
\int\limits_{\omega(x_3)} |\tau(x_1,x_2,x_3)|^2\,dx_1dx_2=\\
=\frac{|\omega(x_3)|}{3\cdot 4}\left(|\bftau_{A^\prime}+\bftau_{B^\prime}|^2+
|\bftau_{B^\prime}+\bftau_{C^\prime}|^2+|\bftau_{A^\prime}+\bftau_{C^\prime}|^2
\right)\,dx_1dx_2
\\
=\frac{|\omega(x_3)|}{6}\left(|\bftau_{A^\prime}|^2+
|\bftau_{B^\prime}|^2+|\bftau_{C^\prime}|^2+\bftau_{A^\prime}\cdot\bftau_{B^\prime}+\bftau_{A^\prime}\cdot\bftau_{C^\prime}+
\bftau_{B^\prime}\cdot\bftau_{C^\prime}\right)=\\
=\frac16 \left(1-\frac{x_3}{h}\right)^4|\Gamma|
\left(|\wh\bftau_{A}|^2+
|\wh\bftau_B|^2+|\wh\bftau_C|^2+\wh\bftau_A\cdot\wh\bftau_B+\wh\bftau_A\cdot\wh\bftau_C+
\wh\bftau_B\cdot\wh\bftau_C\right)
\end{multline*}
We have
\begin{multline*}
\IntO |\tau|^2\,dx=\frac{h}{30}|\Gamma|
\Bigl(\frac{1}{\sin^2\alpha}+\frac{1}{\sin^2\beta}+\frac{1}{\sin^2\gamma}
+\frac{1}{\sin\alpha\sin\beta}\frac{  {\bfeta}\cdot 
 {\bfzeta}}{|\bfeta||\bfzeta|}\\+
\frac{1}{\sin\alpha\sin\gamma}\frac{  {\bfeta}\cdot 
 {\bfsigma}}{|\bfeta||\bfsigma|}+
\frac{1}{\sin\beta\sin\gamma}\frac{  {\bfzeta}\cdot 
 {\bfsigma}}{|\bfzeta||\bfsigma|}.
\Bigr)
\end{multline*}
Since $|\Omega|=\frac13 h|\Gamma|$, we find that
\begin{multline}
\frac{|\Omega|}{|\Gamma|^2}\|\tau\|^2_{2,\Omega}=\\
\frac13 \frac{h^2}{30}
\Bigl(\frac{|\bfeta|^2}{h^2}+\frac{|\bfzeta|^2}{h^2}+\frac{|\bfsigma|^2}{h^2}
+\frac{  {\bfeta}\cdot 
 {\bfzeta}}{h^2}+
\frac{  {\bfeta}\cdot 
 {\bfsigma}}{h^2}+
\frac{  {\bfzeta}\cdot 
 {\bfsigma}}{h^2})=\\
 \frac{1}{90}|
\left(|\bfeta|^2+|\bfzeta|^2+|\bfsigma|^2
+ {\bfeta}\cdot  {\bfzeta}+
  {\bfeta}\cdot  {\bfsigma}+
 {\bfzeta}\cdot  {\bfsigma}\right)
\end{multline}
and
}
\ben
\label{tetrahedron}
\displaystyle C^2_\Gamma\,\leq\,\frac{{\rm d}^2_\Omega}{\pi^2}+
\frac{
|\bfeta|^2+|\bfzeta|^2+|\bfsigma|^2
+{\bfeta}\cdot {\bfzeta}+
{\bfeta}\cdot {\bfsigma}+
{\bfzeta}\cdot {\bfsigma}}
{90}
\een
\vskip3pt
In particular, for the equilateral tetrahedron  with all edges equal to $h$
we have
\be
\bfeta\cdot\bfzeta=\bfeta\cdot\bfsigma=\bfzeta\cdot\bfsigma=\frac12 h^2, \quad
{\rm d}_\Omega=h,
\ee
 and, therefore,
$
C_\Gamma\,\leq\, h\sqrt{\frac{1}{\pi^2}+
\frac{1}{20}}\approx 0.39h.
$

For  the right tetrahedron with nodes $(0,0,0)$, $(h,0,0)$, $(0,h,0)$, $(0,0,h)$
and face $\Gamma=\{x\in \overline \Omega,\,x_3=0\}$,
we have
 ${\rm d}_\Omega=h\sqrt{2}$, $|\eta|=h$, $|\zeta|=|\sigma|=h\sqrt{2}$,
scalar products are equal to $h^2$ and (\ref{tetrahedron}) yields
$
C_\Gamma\,\leq\, h\sqrt{\frac{2}{\pi^2}+
\frac{4}{45}}\approx 0.54 h.
$
Sharp constants $C_\Gamma$ for triangle and tetrahedrons has been recently
evaluated in \cite{MatculevichRepin2015}.
For the right tetrahedron, the constant computed in \cite{MatculevichRepin2015} numerically
 is $C_\Gamma\approx 0.3756 h$.

\subsection{Pyramide}
We can apply (\ref{tetrahedron}) in order to evaluate $C_\Gamma$ for a pyramid
OABCD, which can be divided into two tetrahedrons 
OABC and OACD  (Fig. 2 middle, view from above). Assume that the triangles
ABC and ACD have equal areas and $\Gamma$ is the pyramid basement
ABCD. Then, we can use (\ref{simplex1})
with $\tau$ defined in each tetrahedron as in \ref{subsec:tetrahedron}.
We obtain
\begin{multline}
\label{pyramide}
\displaystyle C^2_\Gamma\,\leq\,\frac{{\rm d}^2_\Omega}{\pi^2}
+
\frac{
2|\bfeta|^2+|\bfzeta|^2+2|\bfsigma|^2+|\bfchi|^2
+2{\bfeta}\cdot {\bfsigma}+
({\bfeta}+\bfsigma)\cdot( {\bfchi}+{\bfzeta})
}{180}.
\end{multline}

\subsection{Prizmatic  cells}
\label{sec:cylindrical}
Consider domains of the form (Fig.   2 right).
$$
\Omega=\{x\in {\mathbb R}^3\,\mid\,(x_1,x_2)\in \Gamma,\quad 0\leq\,x_3\,\leq H(x_1,x_2),\quad H(x_1,x_2)\geq H_{min}\}.
$$
By the same method as in \ref{sec:simplicial} we find that
\ben
\label{cylinder1}
C^2_{\Gamma}\,\leq\,\overline C^2_{\Gamma}:=C^2_{\rm P}+\left(\frac{\mean{H}_\Gamma}{\sqrt 3}+C_{\rm P}\,\kappa\right)^2.
\een
where
$\kappa=\left(\frac{\mean{H}_\Gamma}{H_{min}}-1\right) ^{1/2}$ characterises
variations of the mean height. 

In particular, if $H=const$  (so that $\kappa=0$) and $\Gamma$ is a convex
domain in ${\mathbb R}^{d-1}$, then
\ben
\label{cylinder2}
C^2_{\Gamma}\,\leq\,\frac{{\rm d}^2_\Gamma+H^2}{\pi^2}+\frac{H^2}{3}=
\frac{1}{\pi^2}\left({\rm d}^2_\Gamma+\left(1+\frac{\pi^2}{3}\right)H^2\right).
\een
For a parallelepiped with $\Gamma=(0,a)\times(0,b)$, we know that
the exact value of $C_\Gamma$ is $\frac{1}{\pi}\max\{2H,a,b\}$. In this case
${\rm d}^2_\Gamma=a^2+b^2$ and we can compare it with the upper bound that
follows from (\ref{cylinder1}):
\ben
\label{cylinder3}
\frac{\overline C_{\Gamma}}{C_{\Gamma}}=\frac{\sqrt{a^2+b^2+4.29 H^2}}{\max\{2H,a,b\}}\geq 1.
\een
For the cases where one dimension of $\Omega$ dominates, $\overline C_\Gamma$
is a good approximation of $C^2_{\Gamma}$. If $a=b=H$ (cube), then
we have $\frac{\overline C_{\Gamma}}{C_{\Gamma}}=\frac{\sqrt{6.29}}{2}\approx1.25$. The largest ratio is for $a=b=2H$ ( $\approx 1.75$).

\comment{\color{blue}\small
We set $\bftau=(0,0,-1+\frac{x_3}{H(x_1,x_2)})$, $\dvg\bftau=\frac{1}{H(x_1,x_2)}$.
\begin{multline*}
\IntO |\bftau|^2\,dx=\int\limits_{\Gamma}\int\limits^{H(x_1,x_2)}_0 \left(1-\frac{x_3}{H}\right)^2dx_3dx_1dx_2\\
=\int\limits_{\Gamma} \frac{1}{H^2(x_1,x_2)}\int\limits^{H(x_1,x_2)}_0 (x_3-H)^2dx_3dx_1dx_2\\
=\int\limits_{\Gamma} \frac{1}{3H^2}(x_3-H)^3\mid^H_0 dx_3dx_1dx_2
=
\int\limits_{\Gamma} \frac{H}{3}dx_1dx_2=\frac13|\Omega|\quad[L]^d.
\end{multline*}
Notice that $\mean{\dvg\bftau}_\Omega=\frac{|\Gamma|}{|\Omega|}$,
and  $|\Omega|=\mean{H}_\Gamma|\Gamma|$. We have
\be
\int\limits_{\Gamma}\int\limits^{H}_0 \left(\frac{|\Gamma|}{|\Omega|}-\frac{1}{H(x_1,x_2)}\right)^2dx_3dx_1dx_2=\int\limits_{\Gamma} \left(H\frac{|\Gamma|^2}{|\Omega|^2}-2\frac{|\Gamma|}{|\Omega|}+\frac{1}{H}\right) dx_1dx_2\\
\leq \frac{|\Gamma|^2}{|\Omega|^2}|\Omega|-2\frac{|\Gamma|^2}{|\Omega|}+
\int\limits_{\Gamma} \frac{1}{H} dx_1dx_2\,\leq-\frac{|\Gamma|^2}{|\Omega|}+\frac{|\Gamma|}{H_{min}}=|\Gamma|
\left(\frac{1}{H_{min}}-\frac{1}{\mean{H}_\Gamma}\right)\quad
[L]^{d-2}.
\ee
Dimensionality: $C_{\rm P}=[L]$, $|\Omega|=[L]^d$, $\frac{|\Omega|}{|\Gamma|^2}=[L]^{d}/[L]^{2d-2}=[L]^{2-d}$, \\
$\frac{|\Gamma|}{H}=[L]^{d-1}/[L]=[L]^{d-2}$.    
We have the estimate
\be
C^2_\Gamma\,\leq\,C^2_{\rm P}+\frac{|\Omega|}{|\Gamma|^2}
\left\{\frac{\mu}{3}|\Omega|+\mu^*C^2_{\rm P}|\Gamma|\left(\frac{1}{H_{min}}-\frac{1}{\mean{H}_\Gamma}\right) \right\}\\
=\mean{H}^2_\Gamma\frac{\mu}{3}+C^2_{\rm P}\left(1+\mu^*\frac{|\Omega|}{|\Gamma|}\left(\frac{1}{H_{min}}-\frac{1}{\mean{H}_\Gamma}\right) \right)\\
=C_{\rm P}^2+\mean{H}^2_\Gamma\frac{\mu}{3}+C^2_{\rm P}\left(1+\mu^*\left(\frac{\mean{H}_\Gamma}{H_{min}}-1\right) \right)=C_{\rm P}^2+
\mu\frac{\mean{H}^2_\Gamma}{3}+
\mu^*C^2_{\rm P}\kappa^2,
\ee
where 
$\kappa^2=\frac{\mean{H}_\Gamma}{H_{min}}-1$ is a 
geometrical parameter associated with  "distortion" of $H(x_1,x_2)$.
\color{black}
}

\section{Boundary Poincar\'e inequalities for vector valued functions}

Estimates (\ref{1.10})
and (\ref{1.11}) yield analogous estimates for vector valued
functions in $H^1(\Omega,\Rd)$. 
Let
$ \Omega\in \Rd$ ( $d\geq 1$) be a connected  domain with $N$ plane faces
$\Gamma_i\in {\mathbb R}^{d-1}$.
Assume that we have $d$ unit vectors $\bfn^{(k)}$, (associated with some faces) that
form a linearly independent
system in $\Rd$, i.e.,
\ben
\label{int1}
{\det \bfN\not=0},\qquad  \bfN:=\left\{n^{(i)}_j\right\}\in {\mathbb M}^{d\times d},\quad i,j=1,2,...,d,
\een
where $n^{(i)}_j=\bfn^{(i)}\cdot\bfe_j$ and  $\bfe_i$ denote the Cartesian orts. Then, $\bfv\in H^1(\Omega,\Rd)$ satisfies a Poincar\'e type estimate
provided that it satisfies  zero mean conditions (\ref{int2}).
\begin{Theorem}
\label{InqVector}
If (\ref{int1}) holds and
\ben
\label{int2}
\mean{\bfv\cdot\bfn^{(i)}}_{\Gamma_i} =0\quad i=1,2,...,d,
\een
then
\ben
\label{int3}
\|\bfv\|_\Omega\,\leq\, {\mathds C}(\Omega,\Gamma_1,...,\Gamma_d)\|\nabla \bfv\|_\Omega,
\een
where ${\mathds C}>0$ depends only on geometrical properties of the cell.
\end{Theorem}
\begin{proof}
Assume the opposite. Then, there exists a sequence $\{\bfv_k\}$ such that
$\mean{\bfv_k\cdot\bfn^{(i)}}_{\Gamma_i}=0$  and
\ben
\label{proof1}
\|\bfv_k\|\geq \,k\,\|\nabla \bfv_k\|.
\een
Without a loss of generality we can operate with a sequence of normalised functions, so that
\ben
\label{proof2}
\|\bfv_k\|=1.
\een
Hence,
\ben
\label{proof3}
\|\nabla \bfv_k\|\leq\,\frac{1}{k}\,\rightarrow\,0\;{\rm as}\;k\rightarrow+\infty.
\een
We conclude that there exists a subsequence (for simplicity we omit additional subindexes and keep the notation $\{\bfv_k\}$) such that
\ben
\label{proof4}
&\bfv_k\,\rightharpoonup\,\bfw\;&{\rm in}\; H^1( \Omega,\Rd),\\
\label{proof5}
&\bfv_k\,\rightarrow\,\bfw\;&{\rm in}\;L^2( \Omega,\Rd).
\een
In view of (\ref{proof4}),
\be
0=\lim\inf_{k\rightarrow +\infty} \|\nabla \bfv_k\|\,\geq\, \|\nabla \bfw\|,
\ee
we see that $\bfw\in P^0( \Omega,\Rd)$.
For any face $\Gamma_i$ we have (in view of the trace theorem)
\ben
\label{proof6}
\| \bfv_k-\bfw\|_{2,\Gamma_i}\leq C\left(\|\bfv_k-\bfw\|_{2,\Omega}
+\|\nabla \bfv_k\|_{2,\Omega}\right).
\een
We recall (\ref{proof3}) and (\ref{proof5}) and conclude that the traces of $\bfv_k$ on ${\Gamma_i}$
converge to the trace of $\bfw$. Since  $\bfv_k\cdot\bfn^{(i)}$
have zero means, 
\ben
\label{proof7}
\bfw\cdot \bfn_i |\Gamma_i|=\int\limits_{\Gamma_i}\bfw\cdot \bfn_i\,d\Gamma=0
\een
and
$\bfw$ is orthogonal to $d$ linearly independent vectors, i.e.,
 $\bfw=0$.
On the other hand, $\|\bfw\|=1$. We obtain a contradiction, which shows that the assumption
 is not true.
\end{proof}

We notice that conditions of the Theorem are very flexible with respect to choosing $\Gamma_i$ and vectors $\bfn^{(i)}$ entering the integral type conditions (\ref{int2}). Probably the most interesting case is where $\bfn^{(i)}$ are defined as unit outward normals
 to faces $\Gamma_s$. If $d=2$, then we can also define $\bfn^{(i)}$ 
as unit tangential vectors. Moreover, in the proof it is not essential
that $\bfn^{(i)}$ is strictly related to one face $\Gamma_i$ (only the condition (\ref{int1}) is essential).
For example, if $d=3$ then we can define two vectors as two mutually orthogonal tangential vectors of one face and the third one as a normal vector to another face. Theorem holds for this case as well. Henceforth, for the sake of  definiteness we assume that
$\bfn^{(i)}$ are normal vectors 
or  mean normal vectors (for curvilinear faces) associated with faces $\Gamma_i$,
$i=1,2,...,d$. Possible modifications of the results to other cases are rather obvious.

\subsection{Value of the constant for $d=2$}
Estimates of the constant 
${\mathds C}(\Omega,\Gamma_1,...,\Gamma_d)$ follow from
(\ref{1.10}) and depend on the constants
$C_{\Gamma_i}(\Omega)$.
Now, our goal is to deduce explicit and easily computable bounds of 
${\mathds C}(\Omega,\Gamma_1,...,\Gamma_d)$.

First, we consider a special, but important case where $\Omega$
is a polygonal domain in $ {\mathbb R}^{2}$. Let
$\Gamma_1$ and $\Gamma_2$ be two faces
selected for the interpolation of $\bfv$. The respective
normals
$\bfn^{(1)}=(n^{(1)}_1,n^{(1)}_2)$ and
$\bfn^{(2)}=(n^{(2)}_1,n^{(2)}_2)$ must satisfy the condition 
(\ref{int1}), which means that
\ben
\label{3.4}
\angle(\bfn^{(1)},\bfn^{(2)})=\beta\in (0,\pi).
\een
Let the conditions (\ref{int2}) hold. Then
\ben
\label{3.5}
&&\|n^{(1)}_1v_1+n^{(1)}_2v_2\|^2\leq
C^2_{\Gamma_1}(\Omega)\|n^{(1)}_1\nabla v_1+n^{(1)}_2\nabla v_2\|^2,\\
\label{3.6}
&&\|n^{(2)}_1v_1+n^{(2)}_2v_2\|^2\leq
C^2_{\Gamma_2}(\Omega)\|n^{(2)}_1\nabla v_1+n^{(2)}_2\nabla v_2\|^2.
\een
Introduce the matrix
\be
{\bfT}:=\bfn^{(1)}\otimes\bfn^{(1)}+\!\bfn^{(2)}\otimes\bfn^{(2)}\!
=\!\left(
\begin{array}{cc}
(n^{(1)}_1)^2+(n^{(2)}_1)^2\;&\;n^{(1)}_1n^{(1)}_2+n^{(2)}_1n^{(2)}_2\\
n^{(1)}_1n^{(1)}_2+n^{(2)}_1n^{(2)}_2
&(n^{(1)}_2)^2+(n^{(2)}_2)^2
\end{array}
\right).
\ee
Here and later on $\otimes$ denotes the diadic product of vectors.
Summation of (\ref{3.5}) and (\ref{3.6}) yields
\begin{multline}
\label{3.7}
\IntO {\bfT}\bfv\cdot\bfv dx_1dx_2\\
\leq C^2
\IntO(T_{11} |\nabla v_1|^2+2T_{12}\nabla v_1\cdot\nabla v_2+T_{22} |\nabla v_2|^2)dx_1dx_2,
\end{multline}
where
$$
C=\max\{C_{\Gamma_1}(\Omega);C_{\Gamma_2}(\Omega)\}.
$$
It is easy to see that $\bfT$ is a positive definite matrix. Indeed,
\begin{multline*}
\det(\bfT-\lambda \bfE)\\=((n^{(1)}_1)^2+(n^{(2)}_1)^2-\lambda)
((n^{(1)}_2)^2+(n^{(2)}_2)^2-\lambda)-\bigl(
n^{(1)}_1n^{(1)}_2+n^{(2)}_1n^{(2)}_2\bigr)^2\\
=\lambda^2-2\lambda+ \bigl(n^{(1)}_1n^{(2)}_2-n^{(2)}_1n^{(1)}_2\bigr)^2=\lambda^2-2\lambda+\left(\det{\bfN}\right)^2,
\end{multline*}
where 
\be
{\bfN}:=\left(\begin{array}{c}
\bfn^{(1)}\\
\bfn^{(2)}
\end{array}
\right).
\ee
Hence for any vector $\bfb$, we have
$
\lambda_{1}|{\bfb }|^2\leq\, {\bfT \bfb\cdot \bfb}\,\leq\,
\lambda_{2}|{\bfb }|^2,
$
and
\ben
\label{3.8}
\lambda_{1,2}=1\mp\sqrt{1-\left(\det{\bfN}\right)^2}.
\een
If $\bfn^{(1)}$ and $\bfn^{(1)}$ are orthogonal,
then $\det{\bfN}=1$ and the unique eigenvalue of $\bfN$ is $\lambda=1$.
In this case, the left hand side of (\ref{3.7})
coincides with $\|\bfv\|^2$.
In all other cases $\det{\bfN}<1$ and $\lambda_1<\lambda_2$.

We can always select the coordinate system such that
$$
n^{(1)}_1=1, \;n^{(1)}_2=0,\;
n^{(2)}_1=-\cos\beta,\; n^{(2)}_2=\sin\beta. 
$$
Then,
$$T_{11}=1+\cos^{2}\beta,\; T_{22}=1-\cos^{2}\beta,\;
T_{12}=-\sin\beta\cos\beta,
$$
and the matrix is
\be
{\bfN}:=\left(\begin{array}{cc}
1\;&\;0\\
-\cos\beta\;&\;\sin\beta
\end{array}
\right).
\ee
We see that $ \det{\bfN}=\sin\beta$,
and $\lambda_{1}=1- \Mod{\cos\beta}$.

Consider the right hand side of (\ref{3.7}).
It  is bounded from above by the quantity
\be
I(\bfv):=C^2\IntO\left((T_{11}+\gamma |T_{12}|)|\nabla v_1|^2+(T_{22}+\gamma^{-1} |T_{12}|)|\nabla v_2|^2\right)dx,
\ee
where $\gamma$ is any positive number. We define
$\gamma$ by means of the relation
$T_{11}-T_{22}=(\gamma^{-1}-\gamma)|T_{12}|$, which yields
$\gamma=\frac{1-|\cos\beta|}{\sin\beta}$. Then,
\ben
\label{3.9}
I(\bfv)\,\leq\, (1+|\cos\beta|)\|\nabla\bfv\|^2.
\een

From (\ref{3.7}) and (\ref{3.9}), we find that
\ben
\label{3.10}
\Frame{\|\bfv\|\,\leq\,\max\limits_{i=1,2}
\left\{C_{\Gamma_i}(\Omega)\right\}\,
\sqrt{\frac{1+|\cos\beta|}{1-|\cos\beta|}}\,
\|\nabla \bfv\|.}
\een
This is the Poincar\'e type inequality
for the vector valued function $\bfv$ with
zero mean normal traces on $\Gamma_1$ and $\Gamma_2$. It is worth noting that
for small $\beta$ (and for $\beta$ close to $\pi$) the constant blows up. Therefore, interpolation
operators (considered in Sect. 4) should avoid such situations.
\subsection{Value of the constant for $d\geq 3$}
\label{subsec:constvector}
Now we are concerned with the general case and deduce the estimate valid
for any dimension $d$.

In view of (\ref{int2})
we have
\ben
\label{3.11}
&&\sum\limits^d_{k=1}\|\bfn^{(k)}\cdot\bfv\|^2_{2,\Omega}\,\leq\,C^2\,
\sum\limits^d_{k=1}
\IntO \Bigl(\sum\limits^d_{i=1}
n^{(k)}_i\nabla v_i\Bigr)^2\,d\bfx,
2(\Omega,\Gamma_3)\|n^{(3)}_1\nabla v_1+n^{(3)}_2\nabla v_2+n^{(3)}_3 \nabla v_3\|^2
\een
where 
$$
C=\max\limits_{k=1,2,...,d}\;
\left\{C_{\Gamma_k}(\Omega)\right\}.
$$
In view of the relation
\be
(\bfn^{(k)}\otimes\bfn^{(k)})\bfv\cdot\bfv
=(\bfn^{(k)}(\bfn^{(k)}\cdot\bfv))\cdot\bfv=(\bfn^{(k)}\cdot\bfv)^2,
\ee
 the left hand side of (\ref{3.11}) is
$
\IntO{\bfT}\bfv\cdot\bfv,
$
where 
\ben
\label{3.12}
&&\bfT:=\sum\limits^d_{k=1}\bfn^{(k)}\otimes\bfn^{(k)}.\qquad
\een
If $\bfn^{(k)}$ form a linearly independent system, then
$\bfT$ is a positive definite matrix.
Indeed, $\bfT \bfb \cdot \bfb=\sum\limits^d_{k=1}(\bfn^{(k)}\cdot \bfb)^2$. 
Hence, $\bfT \bfb \cdot \bfb=0$ if and only if $\bfb $ has zero projections to $d$ linearly
independent vectors $\bfn^{(k)}$, i.e.,  $\bfT \bfb \cdot \bfb=0$ if and only if ${\bfb }=0$. Therefore,
\ben
\label{3.13}
\lambda_{1}\,\|\bfv\|^2\,\leq\,
\IntO {\bfT \bfv \cdot \bfv}\,d\bfx,
\een
where $\lambda_1>0$ is the minimal eigenvalue of $\bfT$.

Consider the right hand side of (\ref{3.11}). We have
\begin{multline*}
\IntO\Bigl(\sum\limits^d_{i=1}
n^{(k)}_i\nabla v_i\Bigr)^2d\bfx=
\IntO \sum\limits^d_{i,j=1}n^{(k)}_in^{(k)}_j
\nabla v_i \cdot\nabla v_j d\bfx\\
=
\sum\limits^d_{i,j=1}n^{(k)}_in^{(k)}_j\IntO \nabla v_i \cdot\nabla v_j d\bfx
=\bfn^{(k)}\otimes\bfn^{(k)}:\bfG,
\end{multline*}
where 
$$
\bfG(\bfv):=\{G_{ij}\}, \quad 
G_{ij}(\bfv)=\IntO\nabla v_i\cdot\nabla v_j d\bfx.
$$ 
Hence,
\ben
\label{3.14}
\sum\limits^d_{k=1}
\IntO \Bigl(\sum\limits^d_{i=1}
n^{(k)}_i\nabla v_i\Bigr)^2\,d\bfx=\bfT:\bfG(\bfv)\,\leq\,|\bfT|\,|\bfG(\bfv)|.
\een

Now (\ref{3.11}), (\ref{3.12}), (\ref{3.13}), and
(\ref{3.14}) yield the estimate
\be
\|\bfv\|^2\,\leq\,C^2\frac{1}{\lambda_1}|\bfT|\,|\bfG(\bfv)|
\leq\,C^2\frac{d}{\lambda_1}\,|\bfG(\bfv)|.
\ee
Since $|\bfG(\bfv)|\leq \,\|\nabla\bfv\|^2$, for any $\bfv\in H^1(\Omega,\Rd)$ satisfying
(\ref{int2}) we have
\ben
\label{3.15}
\Frame{\|\bfv\|\,\leq\,C\,\sqrt{\frac{d}{\lambda_1}}\|\nabla\bfv\|.}
\een
In other words, the constant in (\ref{3.15})
can be defined as follows:
\ben
\label{3.16}
\mathds C(\Omega,\Gamma_1,\Gamma_2,...,\Gamma_d)\,=\,
\max\limits_{k=1,2,...,d}
\{C_{\Gamma_k}(\Omega)\}\;
\sqrt{\frac{d}{\lambda_1}},
\een
where $\lambda_1$ is the minimal eigenvalue
of $\bfT$.

For $d=2$ this estimate exposes a slightly worse constant than (\ref{3.10})
with the factor $\sqrt{\frac{2}{1-|\cos\beta|}}$
instead of $\sqrt{\frac{1+|\cos\beta|}{1-|\cos\beta|}}$.
\section{Interpolation of functions}
The classical Poincar\'e inequality (\ref{1.1}) yields a simple interpolation
operator ${\mathbb I}_\Omega: H^1(\Omega)\,
\rightarrow\,P^0(\Omega)$ defined by the relation
$
{\mathbb I}_\Omega w:=\mean{ w }_\Omega$.
In view of (\ref{1.1}), we know that
\ben
\label{4.1}
\|w-{\mathbb I}_\Omega w\|_{2,\Omega}\,\leq\,C_{\rm P}(\Omega)\|\nabla w\|_{2,\Omega},
\een
which means that the interpolation operator
is stable and $C_{\rm P}(\Omega)$ is the respective constant.

Above discussed estimates for functions with zero mean traces yield somewhat
different interpolation operators for scalar and vector valued
functions.  For a scalar valued function $w\in H^1(\Omega)$, we set $\INT_\Gamma(w):=\mean{w}_\Gamma$, i.e., the
interpolation operator uses mean values of $w$ a $d-1$ -- dimensional set
$\Gamma$. Since $\mean{w-\INT_\Gamma w}_\Gamma=0$, 
we  use (\ref{1.10}) and obtain the interpolation estimate
 \ben
 \label{4.2}
\|w-{\mathbb I}_\Gamma w\|_{2,\Omega}\,\leq\,C_\Gamma(\Omega)
\|\nabla w\|_{2,\Omega},
\een
where the constant $C_\Gamma$ appears as the interpolation
constant.
Analogously, (\ref{1.11}) yields an
 interpolation estimate for the boundary trace
\ben
 \label{4.4}
\|w-{\mathbb I}_\Gamma w\|_{2,\Gamma}\,\leq\,C^{\rm Tr}_\Gamma
\|\nabla w\|_{2,\Omega}.
\een
Applying these estimates to cells of meshes we obtain analogous
interpolation estimates for mesh interpolation of scalar functions
with explicit constants depending on character diameter of cells.

For the interpolation of vector valued functions we use (\ref{3.15}) and generalise this idea. 
\subsection{Cells with plane faces}
Define   the operator 
$$
\INT_ {\Gamma_1,\Gamma_2,...,\Gamma_d}: H^1(\Omega,\Rd)\,\rightarrow\, P^0(\Omega,\Rd)
$$ 
that  performs zero order interpolation of a vector valued function $\bfv$ using mean values of normal components on the faces $\Gamma_i$, 
$i=1,2,...,d$. In this case, we set
\ben
\label{4.4}
\int\limits_{\Gamma_i}
\left(\INT_{\Gamma_1,\Gamma_2,...,\Gamma_d}\bfv\right)\cdot\bfn^{(i)}\,d\Gamma=
\int\limits_{\Gamma_i}\bfv\cdot\bfn^{(i)}\,d\Gamma \qquad
i=1,2,...,d.
\een
This condition means that {\em the intrpolant must
preserve integral values of normal flux through $d$ selected
faces}.
In general  we may define several different operators  associated with different collections of faces.
However, once the set of $\Gamma_1,\Gamma_2,...,\Gamma_d$ satisfying (\ref{int1}) has been defined, the operator 
$\INT_ {\Gamma_1,\Gamma_2,...,\Gamma_d}$ uniquely defines the vector
$\INT_ {\Gamma_1,\Gamma_2,...,\Gamma_d}\,\bfv$.
In view of  (\ref{4.4}) and the identity 
$$
\left(\INT_{\Gamma_1,\Gamma_2,...,\Gamma_d}\bfv\right)\cdot\bfn^{(i)}
=(\INT_{\Gamma_1,\Gamma_2,...,\Gamma_d}\bfv)_j\bfe_j\cdot\bfn^{(i)},
$$
we conclude that the components
of the interpolant are uniquely defined by  the system
\ben
\label{4.5}
\sum\limits^{d}_{j=1}
n^{(i)}_j( \INT_ {\Gamma_1,\Gamma_2,...,\Gamma_d}\,\bfv)_j =
\frac{1}{|\Gamma_i|}\int\limits_{\Gamma_i}\bfv\cdot\bfn^{(i)}\,d\Gamma\quad i=1,2,...,d.
\een
Define $\bfw:=\bfv-\INT_{\Gamma_1,\Gamma_2,...,\Gamma_s} \bfv$.
From (\ref{4.4}), it follows that 
\be
\mean{\bfw\cdot\bfn^{(i)}}_{\Gamma_i}=0\qquad i=1,2,...,d.
\ee
Therefore, we can apply Theorem \ref{InqVector} to $\bfw$
and find that
\ben
\label{4.6}
\|\bfw\|_\Omega\,\leq\, {\mathds C}(\Omega,\Gamma_1,...,\Gamma_d)\|\nabla \bfw\|_\Omega.
\een
Since $\nabla \bfw=\nabla \bf v$, (\ref{4.6})
yields the estimate
\ben
\label{4.7}
\Frame{
\|\bfv-\INT_{\Gamma_1,\Gamma_2,...,\Gamma_d} \bfv\|_\Omega\,\leq\, {\mathds C}(\Omega,\Gamma_1,...,\Gamma_d)\|\nabla \bfv\|_\Omega,}
\een
where  ${\mathds C}(\Omega,\Gamma_1,...,\Gamma_d)$
depends on the constants $C_{\Gamma_i}$ (see section \ref{subsec:constvector}).
\subsection{ Cells with curvilinear faces}
\label{subsec:curvilinear}

Let $\Omega$ be a Lipschitz domain
with a piecewise smooth boundary consisting of smooth parts
$\Gamma_1$, $\Gamma_2$,...,$\Gamma_N$ 
(see Fig. \ref{fig:curvilinearcell}). In order to avoid
complicated topological structures (which may lead to difficulties with definitions of "mean normals"), we assume that
all the faces are such that normal vectors can be defined
at almost all points and impose an additional
condition
\be
\bfn_i(x^{(1)})\cdot\bfn_i(x^{(2)})>0\qquad
\forall x^{(1)},x^{(2)}\in \Gamma_i,\quad i=1,2,...,d.
\ee
Then, we can define the mean normal vector associated with $\Gamma_i$: 
$$
\wh\bfn{(i)}:=
\left\{
\frac{1}{|\Gamma_i|}\int\limits_{\Gamma_i}n^{(i)}_1\,d\Gamma,
\frac{1}{|\Gamma_i|}\int\limits_{\Gamma_i}n^{(i)}_2\,d\Gamma,...,
\frac{1}{|\Gamma_i|}\int\limits_{\Gamma_i}n^{(i)}_d \,d\Gamma
\right\}.
$$
\begin{figure}[h]
\label{fig:curvilinearcell}
\centerline{
\includegraphics[width=1.38in]{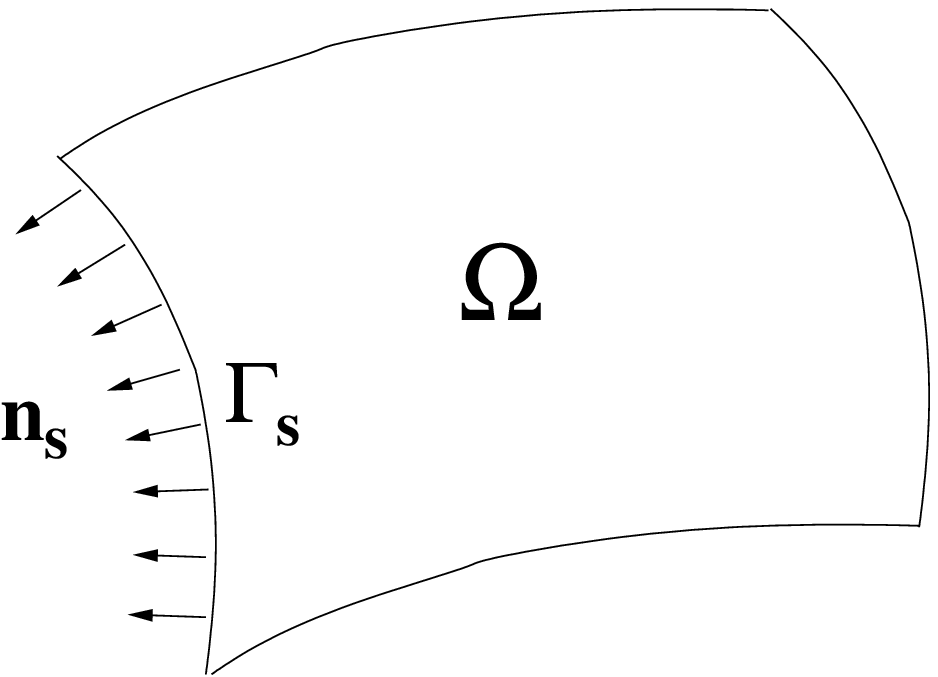}
\qquad\quad\quad\includegraphics[width=1.38in]{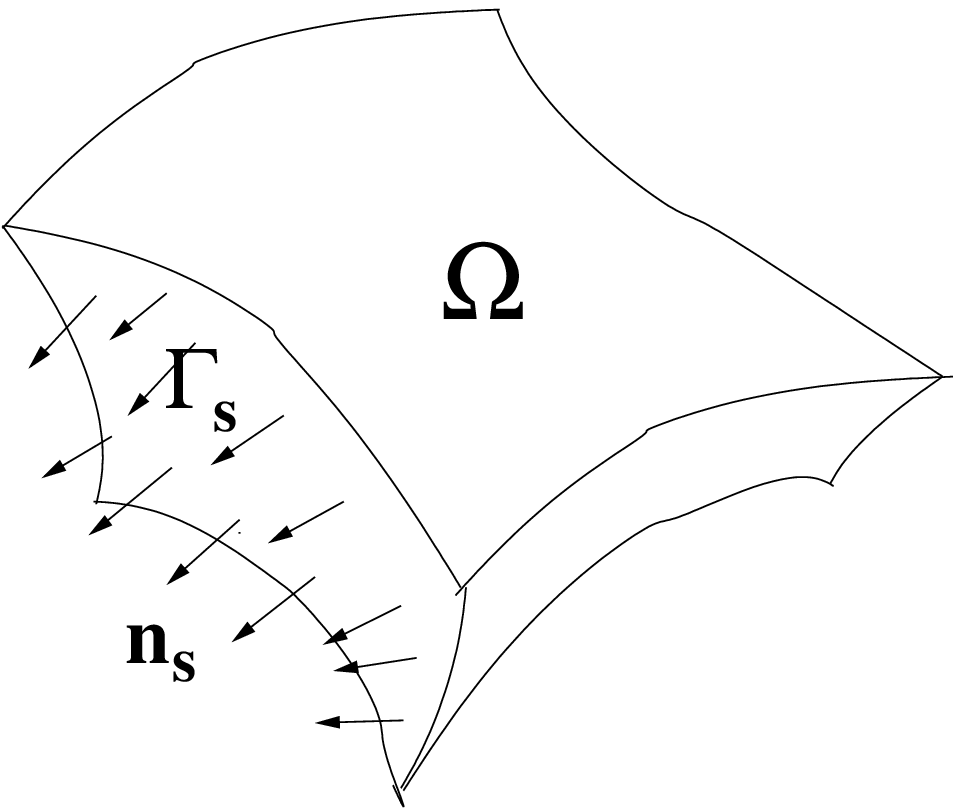}}
\caption{Cells with curvilinear faces in 2D and 3D}
\end{figure}

It is not difficult to verify that Theorem \ref{InqVector} holds
if $\bfN$ is replaced by $\wh\bfN$ formed by mean normal vectors, i.e.,
\ben
\label{4.8}
&&\;\det{\wh\bfN}\not=0,\qquad{\rm where}\qquad 
\wh n^{(i)}_j:=\wh\bfn^{(i)}\cdot \bfe_j,
\een
 and  (\ref{int2}) is replaced by the condition
\ben
\label{4.9}
\mean{\bfv\cdot\wh\bfn^{(i)}}_{\Gamma_i} =0\quad i=1,2,...,d.
\een
In other words, for cells with curvilinear faces the necessary interpolation condition reads as follows:
{\em mean values of normal vectors 
averaged on faces must
form a linearly independent system satisfying (\ref{4.8}).}

The operator
$\INT_{\Gamma_1,\Gamma_2,...,\Gamma_d}\, \bfv$ 
is defined by modifying  the condition (\ref{4.4}). Since
\be
\int\limits_{\Gamma_i}\INT_{\Gamma_1,\Gamma_2,...,\Gamma_d}\bfv\cdot\bfn^{(i)}\,d\Gamma=
\INT_{\Gamma_1,\Gamma_2,...,\Gamma_d}\bfv\cdot\wh\bfn^{(i)}\,\bigl|\Gamma_i\bigr|,
\ee
the interpolant $\INT_{\Gamma_1,\Gamma_2,...,\Gamma_d}\,\bfv$ is defined by the system
\ben
\label{4.9}
&&\sum\limits^{d}_{j=1}
\wh n^{(i)}_j
( \INT_ {\Gamma_1,\Gamma_2,...,\Gamma_d}\,\bfv)_j =
\frac{1}{|\Gamma_i|}\int\limits_{\Gamma_i}\bfv\cdot\bfn^{(i)}\,d\Gamma\quad i=1,2,...,d.
\een

By repeating the same arguments, we obtain the estimate
(\ref{4.7}) for the interpolant 
$\INT_{\Gamma_1,\Gamma_2,...,\Gamma_d}\bfv$.

\subsection{Comparison of interpolation constants for $\INT_\Omega$
and $\INT_\Gamma$}
\subsubsection
{Triangles}
First, we compare five different interpolation operators for the right
triangle with equal legs (see Fig. 1).
For the interpolation operator $\INT_\Omega$ (Fig. 1a)
we have (\ref{1.9}), where (\ref{1.6}) yields the upper
bound of the respective interpolation constant $C_{\rm P}(\Omega)\,\leq\,\sqrt{2}\frac{h}{\pi}\approx 0.4502 h$.
\begin{figure}[h]
\label{trianglecell3}
\centerline{
\includegraphics[width=3in]{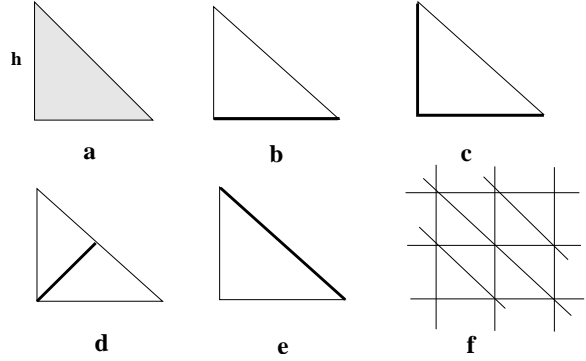}\qquad
}
\caption{Triangular cells}
\end{figure}

Four different operators $\INT_\Gamma$ are generated by setting
zero mean values on one leg (b), two legs (c), median (d),
and hypothenuse (e)
\be
\|w-{\mathbb I}_\Gamma(w)\|_{2,\Omega}\,\leq\;
C_\Gamma(\Omega) h\|\nabla w\|_{2,\Omega}.
\ee
The respective constants follow from Tab. 1.
For (b), $C_\Gamma(\Omega)=\frac{h}{\zeta}\approx  0.4929 h$, for (c)
$C_\Gamma(\Omega)=\frac{h}{\pi}
\approx 0.3183 h$, for (d) and (e)
$C_\Gamma(\Omega)=\frac{h}{\zeta\sqrt{2}}\approx 0.3485 h$.


 
We can use these data and
compare the efficiency of $\INT_\Gamma$ and $\INT_\Omega$ for uniform meshes
which cells are right equilateral triangles (Fig. 1 f). For a mesh with $2nm$ cells, the operator $\INT_\Omega$ uses $2nm$ parameters (mean values on triangles) and provides interpolation with the constant $C_{\rm P}$. The operator $\INT_\Gamma$
using mean values on diagonals (see (e)) has almost the same constant but needs
only $mn$ parameters.

\subsubsection{Squares} Similar results hold for square cells.
For the interpolation operator $\INT_\Omega$ (Fig. 2a) we have the exact
constant $C_{\rm P}=\frac{\pi}{h}$.
The constants for $\INT_\Gamma$ are as follows. For (b), 
$C_\Gamma=\frac{h}{\pi}$, for (c) and (d)
$C_\Gamma=\frac{2h}{\pi}$, and for (e)
$C_\Gamma=\frac{h}{2.869}$.
We see that for a uniform mesh with square cells $\INT_\Gamma$ and $\INT_\Omega$ have the same efficiency if $\Gamma$ is selected as on (d) or (e).
\begin{figure}[h]
\label{trianglecell5}
\centerline{
\includegraphics[width=3in]{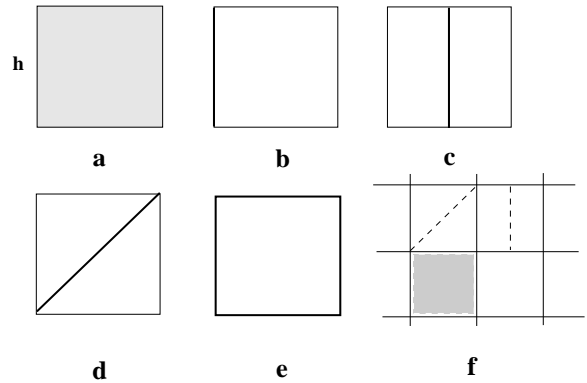}\;
}
\caption{Square cells}
\end{figure}
\subsection{Interpolation on macrocells}
Advanced numerical approximations often operate
with macrocells. Let
 $\Omega$ be a macrocell consisting of $N$ simple subdomains 
 $\omega_i$
 (e.g., simplexes). Let  the boundary $\Gamma$
 consist of faces $\Gamma_i$ (each $\Gamma_i$ is a part
 of some subdomain boundary $\partial\omega_i$).
 For $w\in H^1(\Omega)$ we define $\INT_\Gamma w$ as a piecewise constant function that  satisfies the conditions
 \ben
 \label{4.10}
 \mean{w-\INT_\Gamma w}_{\Gamma_i}=0\qquad i=1,2,...,N.
 \een
  Then, we can apply interpolation operators $\INT_{\gamma_i}$ to any
  subdomain $\omega_i$ and find that for the whole cell
  \begin{multline}
  \label{4.11}
  \|w-\INT_{\Gamma}w\|^2_{2,\Omega}=\sum\limits^N_{i=1}
   \|w-\INT_{\Gamma}w\|^2_{2,\omega_i}\\
   \leq
   \sum\limits^N_{i=1}C^2_{\gamma_i}
   \|\nabla w\|^2_{2,\omega_i}
  \,\leq\, C^2_\Gamma\|\nabla w\|^2_{2,\Omega},
  \end{multline}
  where  
  $C_\Gamma=\max\limits_i\{C_{\Gamma_i}\}$.
  
 Estimates for vector valued functions are derived quite similarly. For example,
 let $d=2$ and $\Omega$ be a polygonal domain
 with $N$ faces.  If $N$ is an odd number, then
 we form out of $\Gamma_i$  a set of $K$ pairs $\{\Gamma^{(l)}_1,\Gamma^{(l)}_2\}$, $l=1,2,..., K$ such that the respective subdomains cover $\Omega$ and for each pair $\bfn^{(l)}_1$ and  $\bfn^{(l)}_2$ satisfy (\ref{int1}).
Then, the interpolant $\INT_\Gamma\bfv$  can be defined as a piece vise constant
field in each pair of subdomains $\omega^{(l)}_1\cup\omega^{(l)}_2$
that satisfies
   \ben
   \label{4.12}
 \mean{(\bfv-\INT_\Gamma\bfv)\cdot\bfn_i}_{\Gamma_i}=0\qquad i=1,2,...,N.
 \een
Analogously to (\ref{4.11}, we obtain
\ben
\label{4.13}
\|\bfv-\INT_\Gamma\bfv\|_{2,\Omega}\,\leq \,{\mathds C}
\|\nabla \bfv\|_{2,\Omega}\qquad \bfv\in H^1(\Omega,{\mathbb R}^{2}),
\een
where
${\mathds C}=\max\limits_{l=1,2,...,K}
{\mathds C}_{\Gamma^{(l)}_1,\Gamma^{(l)}_2}(\omega^{(l)}_1\cup\omega^{(l)}_2)$.

\subsection{Interpolation on
meshes}
Finally, we shortly discuss applications to mesh interpolation.
It is clear that analogous operators $\INT_\Gamma$  can be constructed
for scalar and vector valued functions defined in a bounded Lipschitz
domain $\Omega$, which is covered by a
mesh ${\mathcal T}_h$ with sells $\Omega_i$, $i=1,2,...,M_h$.

Let $\Omega_i$ be Lipschitz domains such that
 $\Omega_i\cup\Omega_j=\emptyset$
if $i\not= j$ and
\ben
\label{5.1}
\overline{\Omega}=\bigcup^{M_h}_{i=1}\overline\Omega_i.
\een
We assume that $c_1 h\,\leq {\rm diam}\Omega_i\leq \,c_2 h$
for all $i=1,2,...M_h$,
where $c_2\geq c_1>0$ and $h$ is a small parameter.  The intersection of $\overline{\Omega}_i$
and  $\overline{\Omega}_j$ is either  empty or  a  face $\Gamma_{ij}$ (which is a Lipschitz domain in ${\mathbb R}^{d-1}$). 
By ${\mathcal E}_h$ we denote the collection of all faces in $\cT_h$.

It is easy to see that a function $w\in H^1(\cD)$ can be
interpolated by a piece vise constant function
on cells  of $\cT_h$ 
if we set
\ben
\label{5.2}
&&\INT_{\cT_h} (w)(x)=\INT_{\Gamma_i}w(x)=\mean{w}_{\Gamma_i}
\qquad {\rm if}\quad x\in\Omega_i.
\een
Here $\Gamma_i$ is a face of $\Omega_i$ selected for the local interpolation
operator. Then,
\ben
\label{5.3}
\|w-\INT_{\cT_h} (w)\|_{2,\Omega}\;\leq\;C(\cT_h)\,\|\nabla w\|_{2,\Omega},
\een
where $C(\cT_h)$ is the maximal constant in inequalities (\ref{1.10}) associated with $\Omega_i$, $i=1,2,...,M_h$. We note that the amount of parameters used in such type interpolation
is essentially smaller than the amount of faces in $\cT_h$.

If $\INT_{\cT_h}$ is constructed by means of averaging on each face
$\Gamma_{ij}$ then (\ref{5.3}) holds with a better constant and the interpolant
$\INT_{\cT_h}w$ possesses an important property: {\em it preserves 
 mean values of $w$.}

Similar consideration is valid for vector valued functions. If we  define the interpolation
operator $\INT_{\cT_h} (\bfv)(x)$ on $\cT_h$ by the conditions
\ben
\label{5.4}
&&\INT_{\cT_h} \bfv\cdot\bfn_{ij}=
\mean{   \bfv\cdot\bfn_{{ij}}   }_{\Gamma_{ij}} 
\qquad \forall\; \Gamma_{ij}\in {\mathcal E}_h,
\een
then
\ben
\label{5.5}
\|\bfv-\INT_{\cT_h} \bfv\|_{2,\Omega}\;\leq\;{\mathds C}(\cT_h)\,
\|\nabla \bfv\|_{2,\Omega},
\een
where ${\mathds C}(\cT_h)$ is the maximal constant in inequalities (\ref{4.13})
used for  $\Omega_i$, $i=1,2,...,N(\cT_h)$.  The interpolant
$\INT_{\cT_h}\bfv$ possesses an important property: {\em it preserves 
mean values of $\bfv\cdot\bfn_{ij}$} on all the faces of $\cT_h$.


\begin{thebibliography}{24}
\bibitem{ArBoFa2002}
D. Arnold, D. Boffi and R. Falk.
Approximation by quadrilateral finite elements
Math. Comp. 71 (2002), 909-922. 


\bibitem{ArBoFa2005}
D. Arnold, D. Boffi and R. Falk.
Quadrilateral H(div) finite elements.
SIAM J. Numer. Anal. 42 (2005), no. 6, 2429--2451. 

\bibitem{BabuskaAziz}
I. Babu\v{s}ka and A. Aziz.
On the angle condition in the finite element method.
SIAM J. Numer. Anal. 13 (1976), no. 2, 214--226. 


\bibitem{Bermudes}
A. Bermudez, P. Gamallo, M. R. Nogueiras,
and R. Rodriguez.
Approximation properties of
lowest-order hexahedral
Raviart--Thomas finite elements.
Comptes Rendus Mathematique, 340 (2005), no 9, 687-692.

\bibitem{BrFo}
F. Brezzi and M. Fortin. Mixed and Hybrid Finite Element Methods. Springer, Berlin, 1991.

\bibitem{Cheng}
S. Y. Cheng.
Eigenvalue comparison theorems and its geometric applications. Math. Z, 143 (1975), no. 3, 28--297.

\bibitem{GiRa}
V.~Girault and P.~A.~Raviart.
{ Finite element approximation of
the Navier--Stokes equations}. Springer-Verlag, Berlin, 1986.

\bibitem{Slosh1}
D. W. Fox and J. R. Kuttler. 
Sloshing frequencies. Z. Angew. Math. Phys., 34 (1983), no. 5, 668--696.
\bibitem{Slosh2}
V. Kozlov and N. Kuznetsov. The ice-fishing problem: the fundamental sloshing frequency versus geometry of
holes. Math. Methods Appl. Sci., 27(2004), no. 3, 289--312.
\bibitem{Slosh3}
V. Kozlov, N. Kuznetsov, and O. Motygin. On the two-dimensional sloshing problem. Proc. R. Soc. Lond.
Ser. A Math. Phys. Eng. Sci., 460 (2004), 2587--2603.


\bibitem{KuzPro2010}
Yu. Kuznetsov and A. Prokopenko. A new multilevel algebraic preconditioner for the diffusion equation in heterogeneous media. Numer. Linear Algebra Appl. 17 (2010), no. 5, 759--769. 
\bibitem{KuzRep2003}
Yu. Kuznetsov and S. Repin.
New mixed finite element method on polygonal and polyhedral meshes,
{ Russ. J. Numer. Anal. Math. Modelling}, Vol. 18(2003), 261--278.
\bibitem{Kuz2006}
Yu. Kuznetsov. Mixed finite element method for diffusion equations on polygonal meshes with mixed cells. J. Numer. Math. 14 (2006), no. 4, 305 -- 315.
\bibitem{Kuz2011}
Yu. Kuznetsov.
Approximations with piece-wise constant fluxes
for diffusion equations.
J. Numer. Math., Vol. 19, No. 4,  309--328 (2011).
 \bibitem{Kuz2014}
Yu. Kuznetsov
Mixed FE method with piece-wise constant fluxes on polyhedral meshes. Russian J. Numer. Anal. Math. Modelling 29 (2014), no. 4, 231 -- 237. 
\bibitem{Kuz2015}
Yu. Kuznetsov.
 Error estimates for the RT0 and PWCF methods for the diffusion equations on triangular and tetrahedral meshes. Russian J. Numer. Anal. Math. Modelling 30 (2015), no. 2, 95 -- 102. 
\bibitem{LaSi}
R. S. Laugesen and B. A. Siudeja. Minimizing Neumann fundamental tones of triangles: an optimal Poincar\'e inequality. J Differential Equations 249 (2010), no. 1, 118--135.

\bibitem{MaNeRe}
O. Mali, P. Neittaanm\"aki and S. Repin. Accuracy Verification Methods. Theory and
Algorithms, Computational Methods in Applied Sciences, 32, Springer, Dordrecht, 2014.

\bibitem{MatNeiRep}
S. Matculevich, P. Neittaanm\"aki, and S. Repin. 
A posteriori error estimates for time-dependent reactiondiffusion
problems based on the Payne--Weinberger inequality. AIMS, 35 (2015), no. 6, 2659 -- 2677.
\bibitem{MatculevichRepin2015}
S. Matculevich and S. Repin. 
Sharp bounds of constants in Poincar\'e-type inequalities for simplicial
domains. ArXiv:1504.03166v5 [math.NA], 2015 (to appear in Comput. Meth. Appl. Math., 2016).


\bibitem{NazRep}
A. Nazarov and S. Repin. Exact constants in Poincar\'e type inequalities
 for functions with zero mean  boundary traces.
 Mathematical Methods in the Applied Sciences, 38 (2015), no. 15,
3195 -- 3207.

 \bibitem{Poincare1984}
 H. Poincar\'e, Sur les equations de la physique mathematique. Rend. Circ. Mat. Palermo 8 (1894).
 \bibitem{PW}
L. E. Payne and H. F. Weinberger. An optimal Poincar\'e inequality
for convex domains, {Arch. Rat. Mech. Anal.}, 5 (1960),
286--292.

\bibitem{ReGruyter}
S. Repin. A posteriori estimates for partial differential
equations. Walter de Gruyter, Berlin, 2008.

\bibitem{ReRJNAMM}
S. Repin. Interpolation of functions based on mean values of boundary traces. { Russ. J. Numer. Anal. Math. Modeling}, 2015.
\bibitem{ReENUMATH}
S. Repin. Estimates of Constants in Boundary Mean Trace
Inequalities and Applications to Error Analysis.
 A. Abdulle et all (eds.) { Numerical Mathematics
 and Advanced Applications}--ENUMATH2013,
Lecture Notes in Computational Science and Engineering,
215-223.
\bibitem{RoTo}
J. E. Roberts and J.-M. Thomas, Mixed and hybrid methods. In: Handbook of Numerical Analy- sis, Vol. II. North-Holland, Amsterdam, 1991,  523 -- 639.
\bibitem{Steklov}
V. A. Steklov. On the expansion of a given function into a series of harmonic functions. Commun.
Kharkov Math. Soc. Ser. 2 5 (1896), 60--73 (in Russian).
\end{thebibliography}
\end{document}